\documentclass[a4paper, 11 pt]{amsart}

\usepackage[english]{babel}
\usepackage[utf8]{inputenc}
\usepackage[T1]{fontenc}
\usepackage{amsmath}
\usepackage{amssymb}
\usepackage{amsthm}
\usepackage{mathtools}
\usepackage{tikz}
\usepackage{tikz-cd}
\usepackage{extpfeil}
\usepackage{amsbsy}
\usepackage[left=3.0cm, right=3.0cm, top=3.50cm, bottom=3.0cm, headheight=14pt]{geometry}
\usepackage[all,cmtip]{xy}
\usepackage{fancyhdr}
\usepackage{verbatim}
\usepackage{graphicx, caption, float, amsmath}
\usepackage[hidelinks]{hyperref}
\usepackage{color}
\usepackage{mleftright}
\usepackage{mathrsfs}
\usepackage{stmaryrd} 
\usepackage{csquotes}

\newtheorem{teor}{Theorem}[section]
\newtheorem{prop}[teor]{Proposition}
\newtheorem{lema}[teor]{Lemma}
\newtheorem{coro}[teor]{Corollary}

\newtheorem{defi}[teor]{Definition}
\newtheorem{rem}[teor]{Remark}
\newtheorem{exam}[teor]{Example}

\newtheorem{nota}[teor]{Notation}
\providecommand{\abs}[1]{\lvert#1\rvert}

\newcommand\enllas{\raise.5pt\hbox{$\boxempty\kern-4.85pt{}^{\tiny\nearrow}$}\kern-2pt}

\author{Ramon Gallardo Campos}
\title[Euler Characteristic of Odd-Dimensional Orbifolds and Their Boundaries]{On the Euler Characteristic of Odd-Dimensional Orbifolds and Their Boundaries}

\DeclareMathOperator{\Fix}{Fix}

\newcommand{\NN}{\ensuremath{\mathbb{N}}} 
\newcommand{\ZZ}{\ensuremath{\mathbb{Z}}} 
\newcommand{\RR}{\ensuremath{\mathbb{R}}} 

\newcommand{\OO}{\ensuremath{\mathcal{O}}}
\newcommand{\SSS}{\ensuremath{\mathcal{S}}}

\newcommand{\PP}{\ensuremath{\mathcal{P}}}
\newcommand{\quot}[2]{
  \mathchoice
  {
    \text{\raise1ex\hbox{$#1$}\!\Big/\!\lower1ex\hbox{$#2$}}}
  {
    #1/#2}
  {
    #1/#2}
  {
    #1/#2}
}

\begin{document}
\thispagestyle{empty}

\begin{abstract}
   In this work we prove that for a compact odd-dimensional orbifold its Euler characteristic is half of the Euler characteristic of its boundary.
\end{abstract}
\maketitle



\tableofcontents
\thispagestyle{empty}

\section{Introduction}

For compact manifolds of odd dimension, there is an intriguing formula: the Euler characteristic of the manifold is half the Euler characteristic of its boundary. Given that orbifolds generalize manifolds, a natural question arises: does this formula hold for orbifolds as well? The answer is yes.

In this paper, we provide two proofs of this result. The first proof is due to Ichiro Satake, who in 1957 stated and proved that the Euler characteristic of an odd-dimensional compact Riemannian orbifold without boundary is zero. From this result, the desired formula follows. However, Satake's proof relies on the Chern-Gauss-Bonnet theorem.

To complement this, we offer a purely topological proof of the formula. Our approach involves decomposing the orbifold into simpler pieces, making the analysis of this formula more tractable. Specifically, we define a stratification of the singular locus and extract neighborhoods of the strata to prove the result.

In the second section, we introduce the basic theory of orbifolds. In the third section, we define the stratification that will be used to prove the main theorem, and in the fourth section, we present the proof.

\section{Orbifold preliminaries}

An orbifold $\OO$ will be a space which can be locally understood as $\mathbb{R}^{n}$ modulo the action of some finite group for some $n \in \mathbb{N}$. This concept should intuitively generalise the concept of manifold, which in the end will be an orbifold with associated finite groups being the trivial group. The formal definition, as stated in \cite{thurston-2023}, is the following:

\begin{defi}
    An orbifold of dimension $n$ or $n$-orbifold $\OO$ consists of a topological Hausdorff space, $|\OO|$, which we call underlying space of the orbifold, equiped with the following structure, called \emph{orbifold atlas}:
    \begin{enumerate}
        \item There exists a countable open covering $\{U_{i}\}_{i}$ of $|\OO|$ closed under finite intersection.
        \item For each $U_{i}$, there exists a finite group $\Gamma_{i}$, an action of $\Gamma_{i}$ on an open subset $\widetilde{U}_{i}$ of $\RR^{n}$ and a homeomorphism $\varphi_{i} \colon U_{i} \longrightarrow \widetilde{U}_{i}/\Gamma_{i}$.
        \item Whenever $U_{i} \subset U_{j}$, there exists an injective homomorphism $f_{ij} \colon \Gamma_{i} \longrightarrow \Gamma_{j}$ and an embedding $\widetilde{\varphi}_{ij} \colon \widetilde{U}_{i} \longrightarrow \widetilde{U}_{j}$ equivariant with respect to $f_{ij}$, i.e. for $\gamma \in \Gamma_{i}$ $\widetilde{\varphi}_{ij}(\gamma x) = f_{ij}(\gamma)\widetilde{\varphi}_{ij}(x)$, satisfying that the following diagram commutes:

        \begin{center}
\begin{tikzcd}
\widetilde{U}_{i} \arrow[dd] \arrow[rrr, "\widetilde{\varphi}_{ij}"]                           &  &  & \widetilde{U}_{j} \arrow[dd]                      \\
                                                                                               &  &  &                                                   \\
\widetilde{U}_{i}/\Gamma_{i} \arrow[rrr, "\varphi_{ij} = \widetilde{\varphi}_{ij}/\Gamma_{i}"] &  &  & \widetilde{U}_{j}/\Gamma_{i} \arrow[dd, "f_{ij}"] \\
                                                                                               &  &  &                                                   \\
                                                                                               &  &  & \widetilde{U}_{j}/\Gamma_{j}                      \\
                                                                                               &  &  &                                                   \\
U_{i} \arrow[rrr, hook] \arrow[uuuu, "\varphi_{i}"]                                            &  &  & U_{j} \arrow[uu, "\varphi_{j}"']                 
\end{tikzcd}
        \end{center}

    \end{enumerate}
\end{defi}

$\widetilde{\varphi}_{ij}$ and $f_{ij}$ are defined up to composition and conjugation, respectively, by elements of $\Gamma_{j}$. Also, although it does not generally hold that when $U_{i} \subset U_{j} \subset U_{k}$ then $\widetilde{\varphi}_{ik}=\widetilde{\varphi}_{jk} \circ \widetilde{\varphi}_{ij}$, there exists an element $\gamma \in \Gamma_{k}$ such that $\gamma \widetilde{\varphi}_{ik}=\widetilde{\varphi}_{jk} \circ \widetilde{\varphi}_{ij}$ and $\gamma f_{ik} \gamma^{-1} = f_{jk} \circ f_{ij}$.

\begin{rem}
    Two atlases give rise to the same orbifold structure if they can be combined to a larger atlas still satisfying the definitions. Therefore, the covering $\{U_{i}\}_{i}$ is not an intrinsic part of the orbifold structure. We say that $(U_{i}, \widetilde{U}_{i},\Gamma_{i},\varphi_{i})$ or just $(\widetilde{U}_{i},\Gamma_{i},\varphi_{i})$ is a chart of the orbifold.
\end{rem}

The information of the local behavior will be strongly related to the following remark. 

\begin{rem}
    Let $\OO$ be an orbifold and $x \in \OO$. Take a local chart $(\widetilde{U}, \Gamma, \varphi)$ such that $x \in \varphi(\widetilde{U})$ and choose $\widetilde{x} \in \widetilde{U}$ such that $\varphi(\widetilde{x}) = x$. It can be proven that the structure of the isotropy subgroup $G_{\widetilde{x}} \subseteq \Gamma$ at $\widetilde{x}$ does not depend on the choice of $\widetilde{U}$ and $\widetilde{p}$, and therefore also does not depend on the choice of the local chart, and is uniquely determined by $x$.
\end{rem}
Last remark motivates the following definition.

\begin{defi}[Local group]
    Let $\OO$ be an orbifold and $x \in \OO$. We define the local group of $\OO$ at $x$, denoted as $\Gamma_{x}$, as the isotropy group of $\widetilde{x}$ where $\widetilde{x} \in \widetilde{U}$ such that $\varphi(\widetilde{x}) = x$ for a local chart $(\widetilde{U}, \Gamma, \varphi)$ such that $x \in \varphi(\widetilde{U})$.
\end{defi} 
By the previous remark, the local group is well defined up to isomorphism.

\begin{rem}
    Note that a point $x \in \OO$ has a local chart $(\widetilde{U}, \Gamma_{x},\varphi)$. Indeed, the local group could be defined as the group associated to a local chart of $x$ where $\varphi(0)=x$.
\end{rem}

A local chart around $x$ with associated group isomorphic to $\Gamma_{x}$ is called a \emph{fundamental} chart.

\begin{rem}\label{point inside fundamental chart implies injection of local groups}
    Let $\OO$ be an orbifold and let $x,y \in \OO$. Assume that exists a fundamental chart of $x$, $(U_{x}, \widetilde{U_{x}}, \Gamma_{x}, \varphi_{x})$, such that $y \in U_{x}$. There exists $\varepsilon > 0$ such that $(U_{y},B(y,\varepsilon), \Gamma_{y}, \varphi_{x_{k}})$ is a fundamental chart of $y$, where $B(y,\varepsilon)$ is a ball centered at a lift of $y$, and $U_{y} \subseteq U_{x}$. This implies that there is a natural injection
        $$\Gamma_{y} \hookrightarrow \Gamma_{x}$$
\end{rem}

\begin{defi}[Regular and singular points]
    We say that $x$ is regular if $\Gamma_{x}$ is trivial, otherwise we say $x$ is singular.
\end{defi}

Thus, the regular points will be the points that have a locally euclidean behaviour (the set of all regular points is in fact an open manifold), whereas the singular ones will have a more complicated local behaviour and therefore will be of our interest.

\begin{defi}[Singular locus]
    The singular locus of an orbifold $\OO$, $\Sigma_{\OO}$, as the set of singular points of $\OO$.
\end{defi}

The notion of submanifold can also be generalised as the following definition shows.

\begin{defi}[Suborbifold]
    A $d$-suborbifold $\OO_{1}$ of an $n$-orbifold $\OO_{2}$ is a subspace $|\OO_{1}| \subset |\OO_{2}|$ locally modelled by $\RR^{d} \subset \RR^{n}$ modulo finite groups.
\end{defi}

Note that $|\OO|$ can be a manifold even though $\OO$ is not. We say $\OO$ is connected if $|\OO|$ is connected, and the same applies to compactness. However, orbifolds with boundary are not those with a manifold with boundary as underlying space; we have to define them with subsets $\widetilde{U}$ of $\RR_{+}^n := \RR^{n-1} \times [0, \infty)$, this is:

\begin{defi}[Orbifold with boundary]
    An orbifold with boundary is a space locally modelled on $\RR_{+}^{n}$ modulo finite groups.
\end{defi}

Several notions can be defined for orbifolds by extending the definition for manifolds. The boundary of $\OO$, denoted by $\partial \OO$, is the set of points $x \in \OO$ such that $x$ has a local prechart $(\phi,U,\widetilde{U},\Gamma,\varphi)$ with $\phi(x) \subset \RR^{n-1} \times \{0\}$. The orbifold $\OO \backslash \partial \OO$ is called the interior of $\OO$ and denoted by $Int (\OO)$. When $\partial \OO = \emptyset$ we say that $\OO$ is \emph{closed} if it is compact and \emph{open} otherwise. One has to be careful to avoid confusion because when $|\OO|$ is a manifold, $\partial |\OO|$, $\partial \OO$ and $|\partial \OO|$ are not necessarily the same. Note that an orbifold without boundary can have a manifold with boundary as its underlying space; we will see several examples in the following sections. 

It follows from the definition that a manifold without boundary is an orbifold whose groups $\Gamma_{i}$ are trivial.

It is also possible to assign an orbifold structure to a manifold with boundary in the following way: the intuitive idea is to double the manifold $M$ by reflecting on $\partial M$. Take $M$ a manifold with boundary, for each $x \in \partial M$ exists a neighborhood modelled on $\RR^{n}/\ZZ_{2}$ where $\ZZ_{2}$ acts on $\RR^{n}$ by reflection on the hyperplane $\{(x_{1}, \ldots , x_{n}) \in \RR^{n} \mid x_{n} \geq 0\}$. The points on $M - \partial M$ are already modelled by neighborhoods on $\RR^{n}$ because their associated groups $\Gamma_{i}$ are all trivial. We will denote this orbifold structure by $mM$. The same idea can be used to assign an orbifold structure to an orbifold with boundary, so we will always work with orbifolds without boundary.

The last construction gives rise to a natural question: Given a topological manifold $M$ and a finite group $\Gamma$, when does $M/\Gamma$ has an orbifold structure ? This is answered in the following proposition (Proposition $13.2.1$ of \cite{thurston-2023}).

\begin{prop}\label{M/Gamma is orbifold}
    Let $M$ be a manifold and $\Gamma$ be a group acting properly discontinuously on $M$, then $M/\Gamma$ has an orbifold structure. 
\end{prop}

Henceforth we will use $M/\Gamma$ to denote the orbifold structure that arises from taking the quotient of $M$ by $\Gamma$. Note that the orbifold structure constructed in the case of a manifold with boundary $M$ is a particular case of this last proposition taking $\Gamma = \ZZ_{2}$ acting by reflection on $M$. Moreover, we have a similar result, which is Theorem $9.19$ of \cite{lee-2012}.

\begin{prop}\label{quotient of manifold by discrete group acting freely ...}
     If $M$ is a connected smooth manifold and $\Gamma$ is a discrete group acting smoothly, freely, and properly on $M$, then the quotient $\quot{M}{\Gamma}$ is a topological manifold and has a unique smooth structure such that $\pi : M \longrightarrow \quot{M}{\Gamma}$ is a smooth covering map.
\end{prop}

Orbifolds which are global quotients by properly discontinuous actions are usually called \emph{good}, those which are quotients by finite groups are \emph{very good}. If they are not quotients of this type then are called \emph{bad} orbifolds. All the examples that have been given so far are good orbifolds, by construction. Let's see some more examples:

\begin{exam}[$p$-teardrop]
    A $p$-teardrop is an orbifold whose underlying space is $S^{2}$ and such that the set of singular points only consists of one point with $\ZZ_{p}$ as local group where $\ZZ_{p}$ acts by rotation of angle $\frac{2\pi}{p}$.
\end{exam}

\begin{figure}[h]
    \centering
    \includegraphics[width=0.2\textwidth]{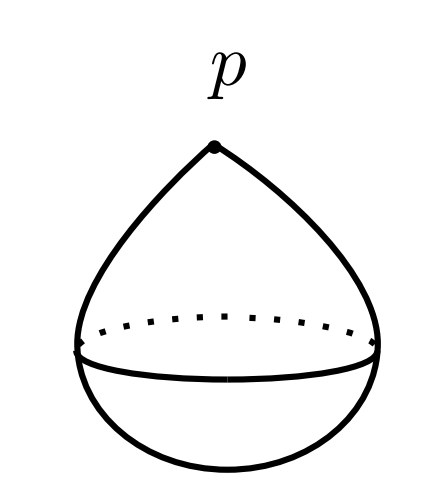}
    \caption{Teardrop of order $p$}
    \label{Teardrop}
\end{figure}

\begin{exam}[$(n,m)$-spindle]
    A $(n,m)$-spindle is an orbifold whose underlying space is $S^{2}$ and such that the set of singular points consists of two points $N,S$ with $\Gamma_{N}=\ZZ_{n}$ and $\Gamma_{S}=\ZZ_{m}$ where $\ZZ_{n}$ and $\ZZ_{m}$ act by rotation of angle $\frac{2\pi}{n}$ and $\frac{2\pi}{m}$ respectively.
\end{exam}

\begin{figure}[h]
    \centering
    \includegraphics[width=0.40\textwidth]{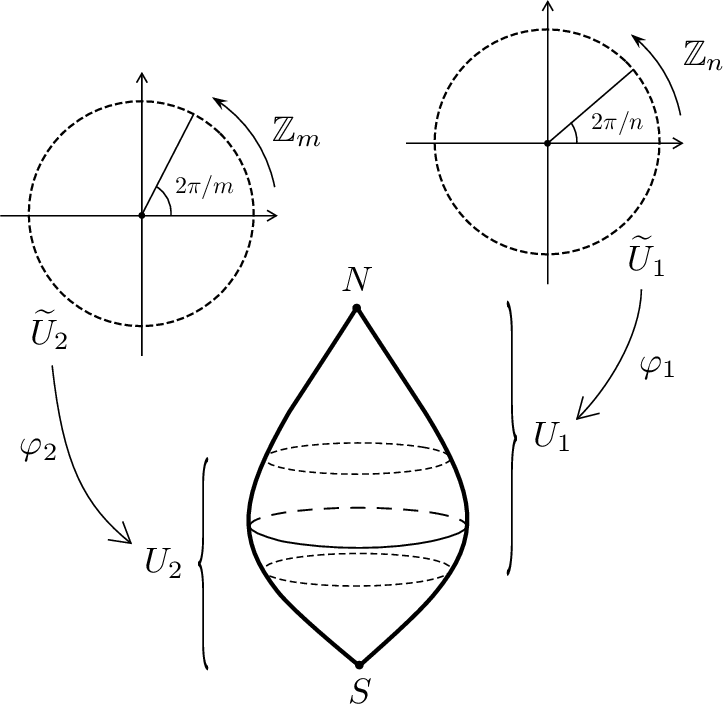}
    \caption{Spindle of order $(n,m)$}
    \label{Football}
\end{figure}

$p$-teardrops and  $(n,m)$-spindles where $n \neq m$ are examples of bad orbifolds. On the other hand, an orbifold with $S^{2}$ as underlying space with three or more singular points is good (as we will see in subsection \ref{section E Char and R Curvature}). An example of this is the $(p,q,r)$-turnover, which is the same as the $(p,q)$-spindle but with one more singular point $D$ with $\Gamma_{D}=\ZZ_{r}$ and $\ZZ_{r}$ acting by rotation of order $r$.

\begin{figure}[h]
    \centering
    \includegraphics[width=0.20\textwidth]{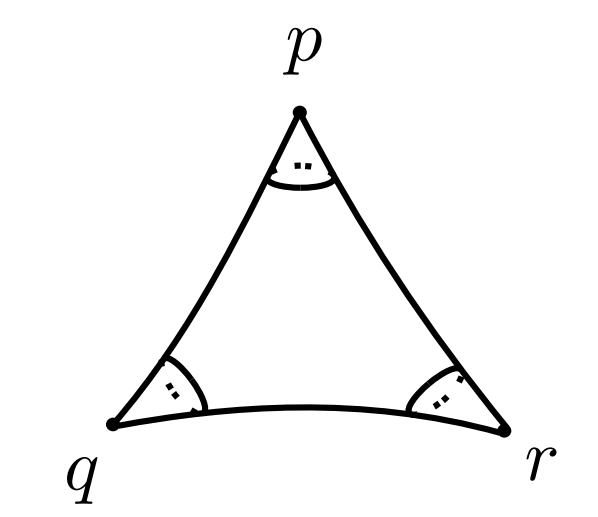}
    \caption{$(p,q,r)$-turnover}
    \label{turnover}
\end{figure}

\subsection{Coverings of an orbifold}

In this section we will define coverings of orbifolds, which will lead us to a definition of the fundamental group of an orbifold.

\begin{defi}
    A covering of an orbifold $\OO$ is a pair $(\PP,\rho)$, where $\PP$ is another orbifold and $\rho:\abs{\PP}\rightarrow\abs{\OO}$ is a surjective map such that for each point $x\in\abs{\OO}$ there is a neighborhood $U$ that admits a chart $(U,\widetilde{U},\Gamma,\phi)$, for which each component $V_i$ of $\rho^{-1}(U)$ admits a chart $(V_i,\widetilde{U},\Gamma_i,\phi_i)$ with $\Gamma_i<\Gamma$. If this holds, we say that $\PP$ is a covering space (or covering orbifold) of $\OO$, or just that $\PP$ covers $\OO$. Sometimes we write $\rho:\PP\rightarrow\OO$ for simplicity.
\end{defi}

\begin{exam}
    If an orbifold has a covering orbifold that is a manifold, we say that it is a good orbifold. Otherwise, we say it is a bad orbifold. This definition coincides with the one given after Proposition \ref{M/Gamma is orbifold}.
\end{exam}

The next lemma is also immediate.

\begin{rem}[Product of good orbifolds is a good orbifold]\label{Product of good orbifolds is a good orbifold}
    Let $\OO_{1}$ and $\OO_{2}$ be good orbifolds with coverings $(M_{1},\rho_{1})$ and $(M_{2},\rho_{2})$, respectively, where $M_{1}$ and $M_{2}$ are manifolds. Then, since the product of coverings is a covering and the product of manifolds is a manifold we conclude that $\OO_{1} \times \OO_{2}$ is a good orbifold with covering $(M_{1}\times M_{2},\rho_{1} \times \rho_{2})$.
\end{rem}

Some common good orbifolds are spherical, discal, annular and toric orbifolds, that are the quotient of, respectively, a sphere, a disk, an annulus and a torus by an isometric action.

\begin{exam}\hspace{0pt}
    By Proposition \ref{M/Gamma is orbifold} if a group $G$ acts properly discontinuously on a manifold $M$ then $M$ is a covering space for $M/G$. In general $M/H$ is a covering space of $M/G$ for each subgroup $H<G$. A particular example is that a cone of order $p$ covers every cone of order $kp$ for all $k\in\NN$.
\end{exam}
\begin{defi}
    The number of sheets of a covering is the cardinality of the preimage of a regular point by $\rho$. If this number is $k<\infty$ we say that the covering is $k$-sheeted.
\end{defi}
\begin{defi}
    A base point of a covering $(\PP,\rho)$ is a regular point $y\in\abs{\PP}$ that is mapped to a regular point in $\abs{\OO}$.
\end{defi}
\begin{defi}
    A universal covering of $\OO$ is a covering $(\PP,\rho)$ such that given any other connected covering $(\PP',\rho')$ and base points $y\in\abs{\PP}$ and $y'\in\abs{\PP'}$ that map to the same point $x\in\abs{\OO}$, there exists a unique covering $(\PP,\pi)$ of $\PP'$ (i.e., $\pi:\abs{\PP}\rightarrow\abs{\PP'}$) such that $\rho=\rho'\circ\pi$ and $\pi(y)=y'$. 
\end{defi}
\begin{teor}
Any connected orbifold $\OO$ admits a universal covering $(\PP,\rho)$ such that $\PP$ is connected. This universal covering is unique up to covering isomorphisms, which are morphisms $f:\PP_1\rightarrow\PP_2$ such that $\rho_1\circ f=\rho_2$ where $\rho_i:\PP_i\rightarrow\OO$, $i=1,2$ are coverings of $\OO$. 
\end{teor}
A proof of this theorem can be found in \cite{thurston-2023} as Proposition 13.2.4.

\subsection{Euler characteristic and Riemann curvature}\label{section E Char and R Curvature}

Many topological and geometrical characteristics of manifolds can be generalized to orbifolds. In this section we will define the Euler characteristic of an orbifold and how can we endow it with a Riemannian metric. From now on we will be interested in smooth orbifolds. An alternative definition of orbifolds in the smooth case, as stated in \cite{porti-2006}, is the following:

\begin{defi}[Smooth orbifold]\label{defi smooth orbifold}
    A smooth n-orbifold $\OO$ is a second countable Hausdorff topological space, $|\OO|$, endowed with a collection $\{(U_{i}, \widetilde{U}_{i},\Gamma_{i},\phi_{i})\}_{i}$, called an atlas, where for each $i$, $U_{i}$ is an open subset of $\OO$, $\widetilde{U}_{i}$ is an open subset of $\RR^{n}$, $\phi_{i} \colon \widetilde{U}_{i} \longrightarrow U_{i}$ is a continuous map, which we call a chart; and $\Gamma_{i}$ is a finite group that acts on $\widetilde{U}_{i}$, that suffice the following conditions:
    \begin{enumerate}
        \item $\OO = \bigcup_{i}U_{i}$
        \item Each $\phi_{i}$ factors through a homeomorphism $\varphi_{i} \colon \widetilde{U}_{i}/\Gamma_{i} \longrightarrow U_{i}$
        \item The charts are compatible in the following sense: for every $x \in \widetilde{U}_{i}$ and $y \in \widetilde{U}_{j}$ with $\phi_{i}(x) = \phi_{j}(y)$, there is a diffeomorphism $\psi$ between a neighborhood $V$ of $x$ and a neighborhood $W$ of $y$ such that $\phi_{j}(\psi(z))= \phi_{i}(z)$ for all $z \in V$.
    \end{enumerate}
\end{defi} 
For convenience, we will always assume that the atlas is maximal.

Now let us backtrack to topological features for a bit:

\begin{defi}
    A triangulation of an orbifold $\OO$ is a triangulation of $|\OO|$, i.e. a homeomorphism between a simplicial complex and $|\OO|$. We say that a triangulation of $\OO$ is compatible if for every interior point of a cell, the group associated to it its the same (we say that it is constant).
\end{defi}
\begin{teor}
    Every smooth orbifold $\OO$ admits a compatible triangulation $T$.
\end{teor}
A proof of this theorem can be found in \cite{choi} as Theorem $4.5.4$.

\begin{defi}\label{defi euler characteristic for orbifolds}
    The Euler characteristic of a compact orbifold $\OO$ with a compatible triangulation $T$ is \[\chi(\OO):=\sum_{\tau\in T}\frac{(-1)^{dim (\tau)}}{n_{\tau}}\]where $n_\tau=\abs{\Gamma_x}$ for any $x$ in the interior of each cell $\tau$.
\end{defi} 
Note that compactness is needed so that the triangulation is finite and the sum is well-defined.

\begin{rem}
    The Euler characteristic of an orbifold is not always an integer.
\end{rem}
\begin{rem}
    This formula does not depend on the triangulation.
\end{rem}

\begin{rem}
    The Euler characteristic is an invariant under homeomorphism.
\end{rem}

\begin{exam}\label{euler characteristic n-tardrop}
    A $n$-teardrop has Euler characteristic $1+\frac{1}{n}$, since it can be built using the usual two charts of the sphere, one quotiented by the trivial group (hence the $1$) and one quotiented by $\ZZ_n$ (hence the $\frac{1}{n}$).
\end{exam}
\begin{prop}\label{k-sheeted covering}
    If $(\PP,\rho)$ is a $k$-sheeted covering of $\OO$, then \[\chi(\PP)=k\chi(\OO)\]
\end{prop}
A proof of this proposition can be found in \cite{caramello} as Proposition 2.4.2.

One also has the following relation.

\begin{prop}\label{characteristic partition}
    Let $\OO$ be a compact orbifold and $X,Y \subseteq \OO$ compact suborbifolds, all of them with compatible triangulations. Then

    $$\chi(X \cup Y) = \chi(X) + \chi(Y) - \chi(X \cap Y)$$
\end{prop}

Now we want to endow orbifolds with a Riemannian metric. First, let's briefly recall what is a Riemannian metric:
\begin{defi}[Riemannian metric]
    A Riemannian manifold is a smooth manifold $M$ equipped with a family of positive-definite inner products $g_p$ on the tangent space $T_pM$ for each $p\in M$. This family of inner products is called a Riemannian metric. 
\end{defi}
Riemannian metrics are used to define geometric notions in manifolds, such as angles, length, areas or curvatures. It can be proved, using partitions of unity, that every smooth manifold admits a Riemannian metric. Similarly, it can be proved for orbifolds. Indeed, we have the followin result (Proposition $4.11$ of \cite{caramello}).
\begin{prop}\label{smooth orbifold admits Riemmaniann metric}
    Any smooth orbifold admits a Riemannian metric.
\end{prop}

\begin{rem}
    The existence of the locally finite orbifold atlas and its subordinate partition of unity is proved similarly to the case of manifolds, one just has to use $\Gamma_i$-invariant functions on each $\widetilde{U}_i$. A proof can be found in \cite{AST} as Lemma $ 4.2.1$.
\end{rem}
A Riemannian metric on a two-orbifold can be used to define the Gauss curvature of the orbifold, which is (in general) different at each point and can be interpreted as whether the orbifold at a given point is locally spherical (positive curvature, called elliptic point), locally saddle-like (negative curvature, called hyperbolic point) or locally flat/cylindrical (zero curvature, called parabolic point). This curvature is an intrinsic property of the two-orbifold, i.e., it does not depend on how it is embedded in an euclidean space.

\begin{prop}\label{every smooth orbifold is locally homeomorphic to R^n/Gamma}
    Every smooth orbifold is locally homeomorphic to $\RR^{n}/\Gamma$ where $\Gamma$ is some finite subgroup of $O(n)$.
\end{prop}

A proof of this Proposition can be found in \cite{thurston-2023}. Local models in dimensions $1$ and $2$ can be easily described using Proposition \ref{every smooth orbifold is locally homeomorphic to R^n/Gamma}.

\begin{rem}[Local models for compact $1$-orbifolds]
    From this last result we have that the only singular points that a compact $1$-orbifold can have is mirror points. 
\end{rem}
\begin{rem}\label{classification of compact 1-orbifolds}
    From last remark we conclude that we only have four types of compact $1$-orbifolds: $[0,1]$, $S^{1}$, the interval with one mirror point ($M_{1}$) and the interval with two mirror points ($M_{2}$).
\end{rem}

In the following all possible singular points in two-orbifolds are described. 

\begin{prop}
    All singular points in a two-orbifold are of one of the following types:
    \begin{enumerate}
        \item Mirror points, whose local group is $\ZZ_2$ and act on $\RR^2$ by reflection in the $y$-axis.
        \item Elliptic or cone points of order $n$, whose local group is $\ZZ_n$ and act on $\RR^2$ by rotations.
        \item Corner reflectors of order $n$, whose local group is $D_n$ and if we write $D_n$ as \[\langle a,b \mid a^2=b^2=(ab)^n=1\rangle\]$a$ and $b$ correspond to reflections in lines meeting at angle $\frac{\pi}{n}$. 
    \end{enumerate}
\end{prop}
\begin{proof}
    All finite subgroups of $O(2)$ are isomorphic to either $\ZZ_n$ or $D_n$, so there are no more possible cases.
\end{proof}
Henceforth, all orbifolds are compact. 

We can now give a more specific formula for the Euler characteristic of two-orbifolds:

\begin{prop} \label{Euler characteristic for two-orbifolds}If $\OO$ is a two-orbifold with $k$ cone points of orders $n_1,\dots,n_k$ and $l$ corner reflectors of orders $m_1,\dots,m_l$ then
    \[\chi(\OO)=\chi(\abs{\OO})-\frac{1}{2}\sum_{i=1}^{l} \left(1-\frac{1}{m_i} \right)-\sum_{i=1}^{k} \left(1-\frac{1}{n_i} \right).\]
\end{prop}
A proof of this formula can be found in \cite{Scott}.

\begin{rem}
    This formula only takes into account the number and types of singular points. Hence, it proves that the Euler characteristic does not depend on the triangulation on two-orbifolds.
\end{rem}

Next definition, which is a generalisation for orbifolds of the definition for manifolds, will play a key role in the proof of Theorem \ref{big theorem}.

\begin{defi}[Neat suborbifold]\label{neat suborbifold}
    Let $\OO$ be an orbifold with boundary and $\SSS$ a closed suborbifold with boundary of $\OO$. We say that $\SSS$ is a \emph{neat suborbifold} if 
    $$\partial \SSS = \SSS \cap \partial \OO$$
\end{defi}

\begin{rem}
    In particular if $\partial \SSS = \emptyset$ then $\SSS$ and $\partial \OO$ are disjoint and therefore $\SSS$ is a suborbifold of $\mathring{\OO}$.
\end{rem}

\begin{rem}
    Let $\OO$ be an orbifold. Let $\SSS$ be a suborbifold such that $\SSS \cap U$ is a neat suborbifold of $U$ for any local chart $U$. Then $\SSS$ is a neat suborbifold of $\OO$.
\end{rem}

\begin{teor}\label{neat submanifold has tubular neighborhood}
    If $M$ is a manifold with boundary and $N$ is a neat submanifold of $M$, then there exists a tubular neighborhood of $N$ in $M$. 
\end{teor}
A proof can be found in Theorem $7.2.12$ of \cite{mukherjee-2015}.

\section{Euler characteristic in odd-dimensional orbifolds}

Henceforth all orbifolds will be smooth. We want to prove the following theorem.

\begin{teor}\label{big theorem}
        Let $\OO$ be a compact odd-dimensional orbifold, then:
        $$\chi(\OO) = \frac{1}{2}\chi(\partial \OO)$$
\end{teor}

\begin{rem}
    Note that last theorem cannot exists for even-dimensional orbifolds since we would have that for an orbifold $\OO$ of any dimension 

    $$\chi(\OO) = \frac{1}{2}\chi(\partial \OO) = \frac{1}{4}\chi(\partial \partial \OO) = \frac{1}{4} \chi(\emptyset) = 0$$

    which is false due to the fact that there exists orbifolds with non-zero Euler characteristic, such as $[0,1]$.
\end{rem}

First, we give some contexts where Theorem \ref{big theorem} holds.

\begin{rem}[Theorem holds in dimension $1$]\label{big theorem holds for dimension 1}
    By Remark \ref{classification of compact 1-orbifolds} we only have four types of compact $1$-orbifolds: $[0,1]$, $S^{1}$, the interval with one mirror point ($M_{1}$) and the interval with two mirror points ($M_{2}$). Appliying formula \eqref{defi euler characteristic for orbifolds} we get that:

        \begin{align*}
\chi([0,1]) &= 2 - 1 = 1      &  \chi(\partial [0,1]) & = \chi(\{0,1\}) = 1 + 1 = 2           \\
\chi(S^{1}) &= 1 - 1 = 0         &  \chi(\partial S^{1})& = \chi(\emptyset) = 0   \\
\chi(M_{1})  &= 1 - 1 + \frac{1}{2} = \frac{1}{2}        &  \chi(\partial M_{1})& = \chi(\{0\}) = 1  \\
\chi(M_{2})&= \frac{1}{2} - 1 + \frac{1}{2} = 0   &  \chi(\partial M_{2})&= \chi(\emptyset) = 0    
\end{align*}

        So we conclude that Theorem \ref{big theorem} holds for dimension $1$.
\end{rem}

In order to prove Theorem \ref{big theorem} we first treat the manifold case. Recall that for a compact $n$-manifold $M$ and $F$ an arbitrary field the Euler characteristic can be defined as:

$$\chi(M) := \sum_{i=0}^{n}(-1)^{i}dim(H_{i}(M,F))$$

It can be proven that the definition of $\chi(M)$ is independent of the field $F$, therefore the above is well defined. The next result (Theorem $3.6$ of \cite{hirsch-1976}) expresses an interesting connection about the behaviour of the Euler characteristic of a manifold and its boundary.

\begin{teor}\label{formula to prove for Manifolds and boundary 0}
    Let $M$ be an odd-dimensional compact smooth manifold without boundary. Then:

$$\chi(M)= 0$$
\end{teor}

From this fact and in the exact same way as proved in Remark \ref{formula works} (which will be seen) we have the following result.

\begin{teor}\label{formula to prove for Manifolds}
    Let $M$ be an odd-dimensional compact smooth manifolds with boundary the following formula holds:

    $$
    \chi(M)= \frac{1}{2}\chi(\partial M)
    $$
\end{teor}

In this section we will provide the necessary tools to give a topological proof of the extension of the formula to compact odd-dimensional smooth orbifolds with boundary.

\begin{prop}
    For very good orbifolds Theorem \ref{big theorem} is a consequence of Theorem \ref{formula to prove for Manifolds}.
\end{prop}
\begin{proof}
    Let $\OO$ be a very good orbifold with $M$ as the finite covering manifold and suppose that the covering $\rho : M \longrightarrow \OO$ is $k$-sheeted with $k > 0$. Then its restriction to the boundary $ \rho_{|\partial M} : \partial M \longrightarrow \partial \OO$ is a $k$-sheeted covering of $\partial \OO$ by $\partial M$. Therefore, by Proposition \ref{k-sheeted covering}, we have that $\chi(M) = k \chi(\OO)$ and $\chi(\partial M) = k \chi(\partial\OO)$. Furthermore by Theorem \ref{formula to prove for Manifolds} we have that 
$$\chi(M)= \frac{1}{2}\chi(\partial M) $$ 
so we conclude that 
$$k\chi(\OO)= \frac{1}{2}\left(k\chi(\partial \OO)\right) $$ 
which is equivalent to
$$  \chi(\OO)= \frac{1}{2}\chi(\partial \OO)$$
\end{proof}

We will use a natural stratification of the singular locus of an orbifold to break the proof into easier ones using the following relation.

Ichiro Satake stated and proved in $1957$ the following theorem (Theorem $4$ of \cite{satake-1957}).

\begin{teor}\label{formula to prove for Orbifolds and boundary 0}
    Let $\OO$ be an odd-dimensional compact smooth orbifold without boundary. Then:

$$\chi(\OO)= 0$$
\end{teor}

\begin{prop}\label{formula works}
    As a consequence of Theorem \ref{formula to prove for Orbifolds and boundary 0} we have that Theorem \ref{big theorem} holds.
\end{prop}

\begin{proof}
    Let $\OO$ be a compact odd-dimensional orbifold with boundary. Then by doubling the orbifold along the boundary we obtain another orbifold without boundary that we will call $2\OO$. Let $\OO_{1}$ be the original half of $2\OO$ and $\OO_{2}$ be the other half. Then it is clear that $2\OO = \OO_{1} \cup \OO_{2}$. We also have that $\OO_{1} \cap \OO_{2} = \partial \OO$, $\OO_{1} \cong \OO \cong \OO_{2}$. Therefore we have that:
$$0 = \chi(2\OO) = 2\chi(\OO) - \chi(\partial \OO)$$ 
which is equivalent to
$$ \chi(\OO) = \frac{1}{2} \chi(\partial \OO)$$
\end{proof}

Nevertheless, the proof of Theorem \ref{formula to prove for Orbifolds and boundary 0}, which was used to prove the generalisation, is based on the Chern-Gauss-Bonnet formula for orbifolds developed by Satake (Theorem $2$ of \cite{satake-1957}), therefore, a topological proof will be constructed in the following pages.
The task to ease the proof using formula \eqref{characteristic partition} will be done by breaking the orbifold into smaller parts where we know that the formula holds, specifically, we will break the orbifold into good suborbifolds. To formalize these ideas we introduce the following definitions.

\begin{defi}[Orbifold decomposition]
    Given a $n$-orbifold $\OO$ we say that a pair of closed $n$-orbifolds $\{\OO_{1}, \OO_{2}\}$ is a decomposition of $\OO$ if 
    \begin{enumerate}
        \item $\OO =\OO_{1} \cup \OO_{2}$
        \item $\OO_{1} \nsubseteq \OO_{2}$ and $\OO_{2} \nsubseteq \OO_{1}$
    \end{enumerate}
     We will say that a decomposition $\{\OO_{i}\}_{i \in A}$ is \emph{neat} if $\OO_{1} \cap \OO_{2}$ is a neat suborbifold of $\OO$ and $\OO_{1} \cap \OO_{2} \subseteq \partial \OO_{1}, \partial \OO_{2}$.
\end{defi}

\begin{rem}\label{neat decomposition iff O1 cap O2 = partial O1 cap partial O2}
    Note that the condition $\OO_{1} \cap \OO_{2} \subseteq \partial \OO_{1}, \partial \OO_{2}$ is equivalent to the condition $\OO_{1} \cap \OO_{2} = \partial \OO_{1} \cap \partial \OO_{2}$. Also note that disjoint decompositions are neat decompositions.
\end{rem}

Neat decompositions of a given orbifold $\OO$ allow to give useful decompositions of its suborbifolds, as the next proposition shows:

\begin{prop}[Boundary decomposition]\label{Boundary decomposition}
    Given a $n$-orbifold $\OO$ with neat decomposition $\{\OO_{1},\OO_{2}\}$ the following equalities hold:

    \begin{enumerate}
        \item $\partial \OO_{i} = (\OO_{i} \cap \partial \OO) \cup (\OO_{1} \cap \OO_{2})$ for $i=1,2$.
        \item $\partial \OO = (\partial \OO \cap \OO_{1}) \cup (\partial \OO \cap \OO_{2})$.
        \item $\partial(\OO_{1} \cap \OO_{2}) = \OO_{1} \cap \OO_{2} \cap \partial \OO$.
    \end{enumerate}
\end{prop}

    \begin{proof}
    We first prove item $i$. By assumption $\OO_{1} \cap \OO_{2} \subseteq \partial \OO_{i}$, also it is clear that $\partial \OO \cap \OO_{i} \subseteq \partial \OO_{i}$ so 
    $$(\partial \OO \cap \OO_{i}) \cup (\OO_{1}\cap \OO_{2}) \subseteq \partial \OO_{i}$$

    We prove the other inclusion. We prove it for $i=1$, case $i=2$ is completely analogous. 
    
    Take $x \in \partial \OO_{1}$. Assume that $x \in \OO_{2}$, then $x \in \OO_{1} \cap \OO_{2}$ and therefore $x \in (\OO_{1} \cap \partial \OO) \cup (\OO_{1} \cap \OO_{2})$.
    
    Now assume that $x \notin \OO_{2}$. Since $\OO_{2}$ is closed, $\OO - \OO_{2}$ is and open subset such that $x \in \OO_{1} - \OO_{2}$. To achieve a contradiction assume that $x \notin \partial \OO$. Then there exists a local chart around $x$ of the form 
    
    $$U_{x} \cong \quot{\RR^{n}}{\Gamma}$$

    such that $U_{x} \subseteq \OO_{1} - \OO_{2} \subseteq \OO_{1}$. Hence we have that $x \in \mathring{\OO_{1}}$, and because $\mathring{\OO_{1}} \cap \partial \OO_{1} = \emptyset$ we arrive to a contradiction with $x \in \partial \OO_{1}$. Hence we conclude that $x \in \partial \OO$ and therefore $x \in \OO_{1} \cap \partial \OO \subseteq (\OO_{1} \cap \partial \OO) \cup (\OO_{1} \cap \OO_{2})$ so
    $$\partial \OO_{1} \subseteq \OO_{1} \cap \partial \OO \subseteq (\OO_{1} \cap \partial \OO) \cup (\OO_{1} \cap \OO_{2})$$

    Item $ii$ is trivial since $\OO = \OO_{1} \cup \OO_{2}$ and hence: $$\partial \OO = \partial \OO \cap \OO = \partial \OO \cap (\OO_{1} \cup \OO_{2}) = (\partial \OO \cap \OO_{1}) \cup (\partial \OO \cap \OO_{2})$$

    Item $iii$ follows by Definition \ref{neat suborbifold}. 
\end{proof}

The next lemma gives a connection between the Euler characteristic of an orbifold and its boundary and the suborbifolds of the decomposition.

\begin{lema}\label{propo para lema importante}
    Let $\OO$ be a compact orbifold of any dimension with neat decomposition $\{\OO_{1},\OO_{2}\}$. If $\chi(\OO_{i}) = \frac{1}{2} \chi(\partial \OO_{i})$ for $i=1,2$ then: 
    $$\chi(\OO)=\frac{1}{2}\chi(\partial \OO) - \frac{1}{2}\chi \bigl(\partial(\OO_{1} \cap \OO_{2}) \bigr) $$
\end{lema}

\begin{proof}
    By Proposition $\ref{Boundary decomposition}$ and Proposition \ref{characteristic partition} we obtain that:

    \begin{equation*}
        \chi(\OO)= \chi(\OO_{1}) + \chi(\OO_{2}) - \chi(\OO_{1} \cap \OO_{2})
    \end{equation*}

    \begin{equation*}
        \chi(\partial \OO) = \chi(\OO_{1}\cap \partial \OO) + \chi(\OO_{2}\cap \partial \OO) - \chi(\OO_{1} \cap \OO_{2} \cap \partial \OO)
    \end{equation*}

    On one hand, by assumption we get that:
    
    $$\chi(\OO) = \frac{1}{2}\Bigl(\chi(\partial \OO_{1}) + \chi(\partial \OO_{2})  \Bigr) - \chi(\OO_{1} \cap \OO_{2})$$

    On the other hand, using decomposition $i)$ from proposition $\ref{Boundary decomposition}$ we have that the latter equals to:

    $$\frac{1}{2}\Bigl(\chi(\OO_{1}\cap \partial \OO) + \chi(\OO_{2}\cap \partial \OO) + 2\chi(\OO_{1}\cap \OO_{2}) - 2 \chi(\OO_{1} \cap \OO_{2} \cap \partial \OO) \Bigl) - \chi(\OO_{1} \cap \OO_{2}) = $$
    $$\frac{1}{2}\Bigl(\chi(\OO_{1}\cap \partial \OO) + \chi(\OO_{2}\cap \partial \OO)\Bigr) - \chi(\OO_{1} \cap \OO_{2} \cap \partial \OO)$$
   
    So we get that the latter equals to:
    $$\frac{1}{2}\chi(\partial \OO) - \frac{1}{2}\chi(\OO_{1} \cap \OO_{2} \cap \partial \OO)$$

    And by item \emph{iii} of Proposition \ref{Boundary decomposition} the latter is equivalent to

    $$\frac{1}{2}\chi(\partial \OO) - \frac{1}{2}\chi \bigl(\partial(\OO_{1} \cap \OO_{2}) \bigr)$$

    And hence the formula is proven.

    \end{proof}

    Next lemma will allow to give an inductive proof for Theorem \ref{big theorem} by giving neat decompositions.    

    \begin{lema}\label{lema importante}
         Let $n$ be odd and $\OO$ a compact $n$-dimensional orbifold with neat decomposition $\{\OO_{1},\OO_{2}\}$. Assume that Theorem \ref{big theorem} holds for $(n-2)$-orbifolds. If $\chi(\OO_{i}) = \frac{1}{2} \chi(\partial \OO_{i})$ for $i=1,2$ then: 
    $$\chi(\OO)=\frac{1}{2}\chi(\partial \OO)$$
    \end{lema}

    \begin{proof}
        Since $\OO_{1} \cap \OO_{2} \subseteq \partial \OO_{i}$ for $i=1,2$ and $\partial(\OO_{1} \cap \OO_{2}) = \OO_{1} \cap \OO_{2} \cap \partial \OO$ we conclude that $\OO_{1} \cap \OO_{2} \cap \partial \OO$ is a $(n-2)$-orbifold such that $\partial (\OO_{1} \cap \OO_{2} \cap \partial \OO) = \partial \partial(\OO_{1} \cap \OO_{2}) = \emptyset$, so by assumption we get:
        $$\chi(\OO_{1} \cap \OO_{2} \cap \partial \OO) = \frac{1}{2} \chi(\partial(\OO_{1} \cap \OO_{2} \cap \partial \OO)) = 0$$

        Additionally, by Lemma \ref{propo para lema importante} we have:

        $$\chi(\OO)=\frac{1}{2}\chi(\partial \OO) - \frac{1}{2}\chi(\OO_{1} \cap \OO_{2} \cap \partial \OO) $$

        Hence the claim follows.
    \end{proof}

In the following we prove Theorem \ref{big theorem} for dimension $3$. The proof contains the basic ideas that will apply to prove the general case. We will use that if $A$ is a vertex, a circle or a segment of the singular locus then there exists $\varepsilon > 0$ such that

$$N_{\varepsilon}(A) := \{x \in \OO \mid d(x,A) < \varepsilon\}$$

is a neighborhood of $A$ that is a very good orbifold and such that

$$S_{\varepsilon}(A):= \{x \in \OO \mid d(x,A) = \varepsilon\}$$

is a neat suborbifold of $\OO$, where $d$ is the distance function of $\OO$. This last result will follow from Lemma \ref{S-epsilon is neat submanifold} and Theorem \ref{tubular neighborhood theorem for minimal strats}.

    \begin{exam}\label{big teorema para caso n=3}
        Let $\OO$ be a compact $3$-orbifold, then:
        $$\chi(\OO) = \frac{1}{2}\chi(\partial \OO)$$
    \end{exam}

    \begin{proof}
        We will extract parts of the singular locus until we obtain a manifold, then by Remark \ref{big theorem holds for dimension 1} and the iterative application of Lemma \ref{lema importante} the result will follow.

        Let $A$ be a vertex, a circle or a segment. Then there exists $\varepsilon > 0$ such that $N_{\varepsilon}(A)$ is an open neighborhood of $A$ that is a very good orbifold and $S_{\varepsilon}(A)$ is a neat suborbifold of $\OO$. We can then take 
        
        $$\OO_{1}:= \overline{N_{\varepsilon}(A)} \hspace{20pt} \OO_{2}:= \OO - N_{\varepsilon}(A)$$

        It is clear that $\OO_{1} \cap \OO_{2} \subseteq \partial \OO_{1}, \partial \OO_{2}$ and that $\overline{N_{\varepsilon}(A)}$ is a very good orbifold. Moreover we have that 

        $$\OO_{1} \cap \OO_{2} = S_{\varepsilon}(A)$$

        so we conclude that $\{\OO_{1}, \OO_{2} \}$ is a neat decomposition of $\OO$.

        Note that the singular locus can have vertices, segments and circles. Assume that the singular locus has vertices. Let $v \in \sum_{\OO}$ be a vertex. By the procedure explained there is a neat decomposition of $\OO$, $\{\OO_{1}, \OO_{2}\}$, such that the $v$ is contained in $\OO_{1}$. Note that the same procedure can be done with $\OO_{2}$. We can continue this process until all vertices are extracted in a finite number of steps. Indeed we can do the same for segments and circles until all parts of the singular locus are extracted in finitely many steps, obtaining a manifold. Therefore the result follows by Remark \ref{big theorem holds for dimension 1} and by applying Lemma \ref{lema importante} iteratively.

        If the singular locus does not have vertices the singular locus is a disjoint union of circles. Then by making use of the same procedure to extract the circles the claim follows.
    \end{proof}

    \subsection{Stratification of an orbifold}

    The idea of the latter proof is to dissect the orbifold into very good orbifolds by removing certain parts of the singular locus. Note that a classification of the subgroups of $O(3)$ was used in the process. Nevertheless, it is not needed but only some properties of the singular locus as it will be shown in the following pages. Indeed, we can adapt this proof to the $n$-dimensional case. For this, we first have to define a stratification that will allow to do this task without knowing the classification of subgroups of $O(n)$. We begin by defining a local equivalence relation.

    \begin{rem}
        Let $\OO$ be an orbifold and $\{U_{\alpha}\}_{\alpha}$ be a countable open covering closed under finite intersection of $\OO$. Take $(U_{\lambda}, \widetilde{U_{\lambda}}, \Gamma_{\lambda}, \varphi_{\lambda})$ and  $(U_{\mu}, \widetilde{U_{\mu}}, \Gamma_{\mu}, \varphi_{\mu})$ to be two charts of $\OO$ such that $U_{\lambda} \cap U_{\mu} \neq \emptyset$. Because $\{U_{\alpha}\}_{\alpha}$ is closed under finite intersection, there exists a group $\Gamma_{\lambda,\mu}$ and a diffeomorphism $\varphi_{\lambda,\mu}$ such that  $(U_{\lambda} \cap U_{\mu}, \widetilde{U_{\lambda} \cap U_{\mu}}, \Gamma_{\lambda,\mu}, \varphi_{\lambda,\mu})$ is a chart. 
    \end{rem}

    \begin{defi}[Local relation]
        Let $\OO$ be a smooth orbifold and let $x,y \in \OO$, we will say that $x$ and $y$ are locally related, denoted as $x \overset{l}{\sim} y$, if there exists local charts $(U_{x}, \widetilde{U_{x}}, \Gamma_{x}, \varphi_{x})$, $(U_{y}, \widetilde{U_{y}}, \Gamma_{y}, \varphi_{y})$ (where $\Gamma_{x}$ and $\Gamma_{y}$ denote the local groups of $x$ and $y$ respectively) around $x$ and $y$, respectively such that $U_{x} \cap U_{y} \neq \emptyset$ and there exists a diagram

        \begin{center}
            \begin{tikzcd}
\Gamma_{x} &  & {\Gamma_{x,y}} \arrow[rr, "f_{y}", hook] \arrow[ll, "f_{x}"', hook'] &  & \Gamma_{y}
\end{tikzcd}
\end{center}

such that the injective homomorphisms $f_{x}$ and $f_{y}$ of the definition of orbifold are isomorphisms, $\varphi_{x}$ is $f_{x}$-equivariant, $\varphi_{y}$ is $f_{y}$-equivariant and $\Gamma_{x,y}$ is such that 
$$(U_{x} \cap U_{y}, \widetilde{U_{x} \cap U_{y}}, \Gamma_{x,y}, \varphi_{x,y})$$
is a chart.
        
    \end{defi}

    This local relation can be globalised as follows.

    \begin{defi}
        Let $\OO$ be a smooth orbifold and $x,y \in \OO$. Then $x$ and $y$ are related, denoted as $x \sim y$, if there exists a finite sequence $\{a_{i}\}_{i \in \{1, \ldots, n\}} \subset \OO$ such that $x \overset{l}{\sim} a_{1} \overset{l}{\sim} a_{2} \overset{l}{\sim} \ldots \overset{l}{\sim} a_{n} \overset{l}{\sim} y$.
    \end{defi}

    It is clear that $\sim$ is an equivalence relation. With this defintion a stratification for orbifolds, and in particular also for the singular locus, arises. Let $\OO$ be a smooth $n$-orbifold and $x,y \in \OO$. We say that $x$ and $y$ belong to the same stratum if $x \sim y$. We define $Strat(\OO)$ to be the set of all strata of $\OO$.

    \begin{rem}\label{strata is disjoint}
        Note that since strata are defined as equivalence classes, two different strata are disjoint.
    \end{rem}

    \begin{rem}[Stratum is arc-connected]\label{Stratum is connected}
        By construction every stratum is arc-connected, hence connected.
    \end{rem}

    Some properties of this stratification are studied in the following. 

    \begin{nota}
        Let $S \subset \OO$ to be a stratum. Then $\forall x \in S$ $\Gamma_{x} \cong \Gamma$ for some $\Gamma < O(n)$. In this situation we put $G(S)=\Gamma$ to denote that the constant local group associated is $\Gamma$.
    \end{nota}
    Note that a stratum is determined by the local group but different strata can have the same local group associated. Indeed, let $\OO$ be a $n$-orbifold, then we have the following a stratification for the singular locus:
    $$\Sigma_{\OO} = \bigcup_{\{e\} \neq H \in \mathcal{G}(\OO)}\{S \in Strat(\OO) \mid G(S) = H\}$$

    where $\mathcal{G}(\OO)$ is the set of local groups of $\OO$.

    The behaviour of the closure of a stratum is strongly related to the behaviour of the stratum, as the following proposition states.

    \begin{lema}\label{G(S) hookrightarrow Gamma_x for x in closure(S)-S}
        Let $S \in Strat(\OO)$. If $\overline{S} \backslash S \neq \emptyset$, for every $x \in \overline{S} \backslash S$ the natural map
        $$G(S) \hookrightarrow \Gamma_{x}$$
        is injective but not surjective.
    \end{lema}

    \begin{proof}
        Take $x \in \overline{S} \backslash S$. By construction for every neighborhood $N_{x}$ of $x$ we have $N_{x} \cap S \neq \emptyset$. Let $\widetilde{S}$, $\widetilde{N_{x}}$ and $\widetilde{x}$ to be lifts of $S$, $N_{x}$ and $x$, respectively. Then we also have that $\widetilde{S} \cap \widetilde{N_{x}} \neq \emptyset$ and that $\widetilde{S} \cap \widetilde{N_{x}} \subset \widetilde{S} \overset{\psi}{\hookrightarrow} \Fix(G(S))$. Hence 
        $$\psi(\widetilde{N_{x}}) \cap \Fix(G(S)) \neq \emptyset$$
        
        Then $\psi(\widetilde{x}) \in \overline{\Fix(G(S))}$. But $\overline{\Fix(G(S))} = \Fix(G(S))$ since $\Fix(G(S))$ is a closed set. Therefore $\psi(\widetilde{x}) \in \Fix(G(S))$ and $\overline{\widetilde{S}} \backslash \widetilde{S} \hookrightarrow \Fix(G(S))$, hence $G(S) \hookrightarrow \Gamma_{x}$. Moreover, $\widetilde{S} \hookrightarrow \Fix(G(S))$ so $\overline{\widetilde{S}} \backslash \widetilde{S} \hookrightarrow \Fix(G(S))$. Also it is clear that $\Gamma_{x}$ cannot be embeded into $ G(S)$ since $x \notin S$ and therefore $\Gamma_{x} \ncong G(S)$, so $G(S) \hookrightarrow \Gamma_{x}$ is not surjective.
    \end{proof}

    \begin{coro}\label{there does not exist injection of group of frontier to G(S)}
        Let $\overline{S} \backslash S \neq \emptyset$ and $x \in \overline{S} \backslash S$. Then there does not exist any injection
        $$\Gamma_{x} \hookrightarrow G(S)$$
    \end{coro}
    \begin{proof}
        By Lemma \ref{G(S) hookrightarrow Gamma_x for x in closure(S)-S} there is an injection $G(S) \hookrightarrow \Gamma_{x}$ that is not surjective, hence 
        $$|G(S)| < |\Gamma_{x}|$$

        Whence we conclude that it cannot exist any injection 
        $$\Gamma_{x} \hookrightarrow G(S)$$
    \end{proof}

    \begin{prop}\label{closure(S) - S is compact}
        Let $S \in Strat(\OO)$. Then $\overline{S} \backslash S$ is compact, and hence a closed set.
    \end{prop}

    \begin{proof}
        If $\overline{S} \backslash S = \emptyset$ it is clear. Assume $\overline{S} \backslash S \neq \emptyset$. Let $\{x_{n}\}_{n \in \NN} \subseteq \overline{S} \backslash S $ be a sequence. Because $\overline{S}$ is compact this sequence converges to $x$ at $\overline{S}$. Recall that given a neighborhood of $x$ there exists some $k \in \NN$ such that $x_{i}$ belongs to that neighborhood for $i \geq k$. For instance, take $(U_{x}, \widetilde{U_{x}}, \Gamma_{x}, \varphi_{x})$ a fundamental chart of $x$. Then there exists some $k \in \NN$ such that $x_{i} \in U_{x}$ for $i \geq k$. By Remark \ref{point inside fundamental chart implies injection of local groups} we have
        
        $$\Gamma_{x_{k}} \hookrightarrow \Gamma_{x}$$
        
        Now assume that $x \in S$. By definition we conclude that we have an injection: 
        
        $$\Gamma_{x_{i}} \hookrightarrow \Gamma_{x} \cong G(S)$$
        
        for $i \geq k$, in contradiction with Corollary \ref{there does not exist injection of group of frontier to G(S)}. Therefore we conclude that $x \in \overline{S} \backslash S$ and hence $\overline{S} \backslash S$ is compact.
    \end{proof}

    We can define a partial order on $Strat(\OO)$. 
    
    \begin{defi}
        Let $\OO$ be a $n$-orbifold and let $S_{1}, S_{2} \in Strat(\OO)$, then we define a binary relation by stating that $S_{1} \leq S_{2}$ if $S_{1} \subseteq \overline{S_{2}}$.
    \end{defi}
    It is clear that the relation is reflexive and transitive so it is a preorder. Moreover, the relation is a partial order.

    \begin{lema}
        The relation $\leq$ is antisymmetrical.
    \end{lema}

    \begin{proof}
        Take $S_{1},S_{2} \in Strat(\OO)$ such that $S_{1} \leq S_{2}$ and $S_{2} \leq S_{1}$, then by Proposition \ref{G(S) hookrightarrow Gamma_x for x in closure(S)-S} we have:
    $$S_{1} \subseteq \overline{S_{2}} = S_{2} \cup \overline{S_{2}} \backslash S_{2} \hspace{20pt} S_{2} \subseteq \overline{S_{1}} = S_{1} \cup \overline{S_{1}} \backslash S_{1}$$
    
    where for every $x \in \overline{S_{2}} \backslash S_{2}$ $G(S_{2}) \hookrightarrow \Gamma_{x}$ and for every $y \in \overline{S_{1}} \backslash S_{1}$ $G(S_{1}) \hookrightarrow \Gamma_{y}$ are the natural injections. If $S_{1} \subseteq S_{2}$ then $G(S_{1})=G(S_{2})$ and if $S_{1} \subseteq \overline{S_{2}} \backslash S_{2}$ we have that $G(S_{2}) < G(S_{1})$, so we conclude that $G(S_{2}) \leq G(S_{1})$. The same argument shows that $G(S_{1}) \leq G(S_{2})$ so we have that $G(S_{1})=G(S_{2})$. Then $S_{1} \cap \overline{S_{2}}\backslash S_{2} = \emptyset$ so $S_{1} \subseteq S_{2}$ and the same argument shows that $S_{2} \subseteq S_{1}$ so we conclude that $S_{1} = S_{2}$ and hence the relation is antisymmetrical.
    \end{proof}

    Provided that is a partial order we have a partition of $Strat(\OO)$ in chains determined by the binary relation $\leq$. 

    \begin{defi}[Chain and minimal stratum]\label{defi chain and minimal stratum}
        Let $S_{1} < \ldots < S_{n}$ where $S_{i} \in Strat(\OO)$ for $i=1, \ldots , n$, then we say that $(S_{1}, \ldots , S_{n})$ is a $n$-chain.

        \begin{enumerate}
            \item If $s = (S_{1}, \ldots , S_{n})$ is a $n$-chain we denote its $i$-th term by $s[i]$, this is $s[i] = S_{i}$.
            \item We say that the chain is \emph{upper complete} if $\nexists S \in Strat(\OO)$ such that $S_{n} < S$ and \emph{lower complete} if $\nexists S' \in Strat(\OO)$ such that $S' < S_{1}$.
            \item We say that the chain is \emph{full} if $\forall S_{i}$ $\nexists S' \neq S_{i},S_{i+1}$ such that $S_{i} < S' < S_{i+1}$.
            \item The chain is \emph{complete} if it is upper and lower complete and full. We denote by $cch(\OO)$ the set of complete chains of an orbifold $\OO$.
            \item A minimal element of a complete chain is called \emph{minimal stratum}.
        \end{enumerate}
    \end{defi}

    \begin{exam}[Complete chains of $\quot{D^{3}}{T_{12}}$]\label{complete chains of D^3 quot T_12}
        Take $\OO=\quot{D^{3}}{T_{12}}$ where $T_{12}$ is the tetrahedral group, the group of orientation preserving symmetries of the tetrahedron. We know that the singular locus is a trivalent graph where two edges, namely $E_{1}$ and $E_{2}$, have $G(E_{1})=\ZZ_{3}=G(E_{2})$, and the other edge, $E_{3}$, has $G(E_{3})=\ZZ_{2}$, with $G(E_{i})$ acting by rotation for $i=1,2,3$ and the origin, $O$, that has $G(O)=\Gamma_{O}=T_{12}$. Then we have 
        $$cch(\OO)=\{(O,E_{1}), (O,E_{2}), (O,E_{3})\}$$

        \begin{figure}[h]
    \centering
    \includegraphics[width=0.20\textwidth]{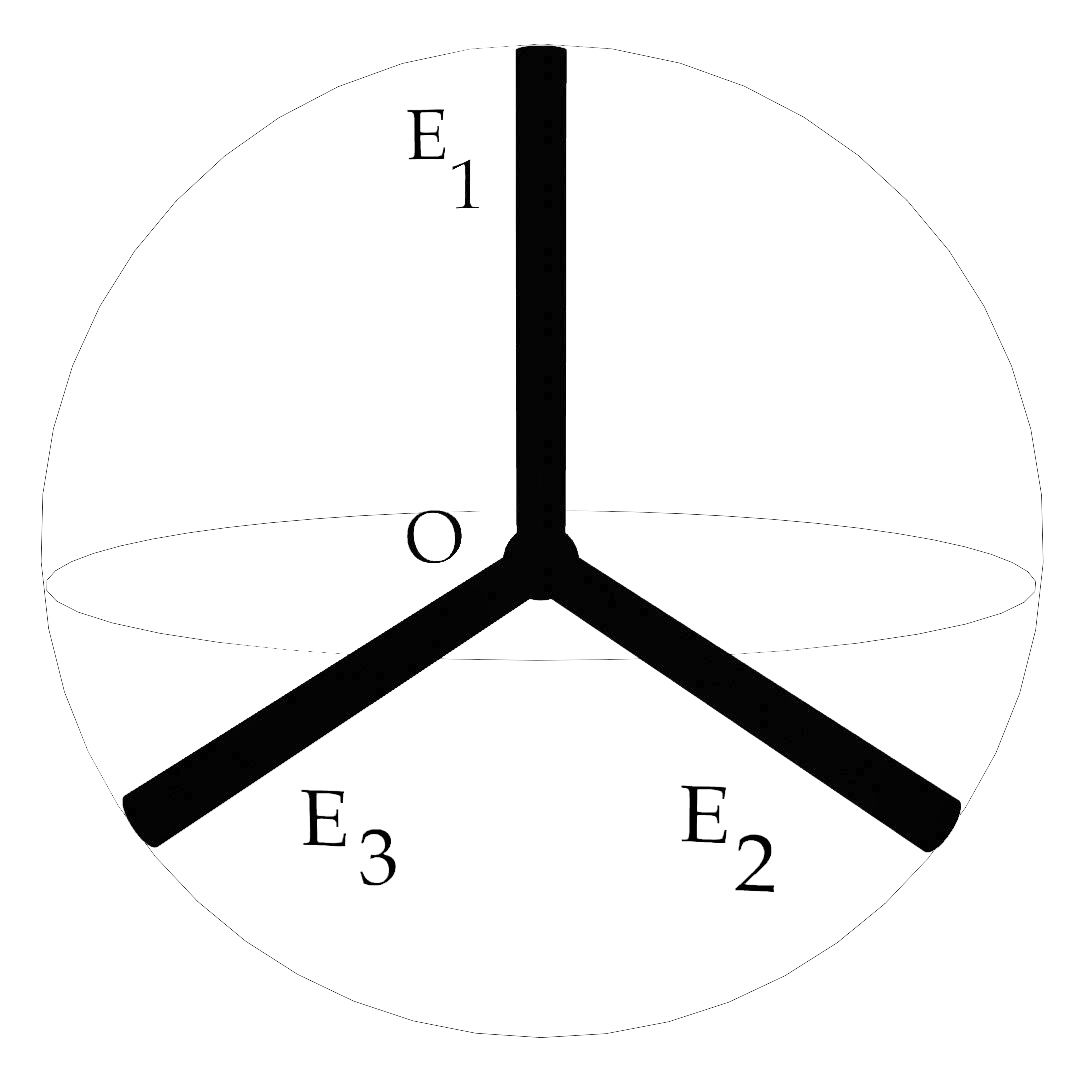}
    \caption{$\quot{D^{3}}{T_{12}}$}
    \label{quot{D^{3}}{T_{12}}}
\end{figure}

    \end{exam}

     As a corollary of Lemma \ref{G(S) hookrightarrow Gamma_x for x in closure(S)-S} we have the following.

     \begin{coro}\label{minimal strata are closed sets}
         Let $S \in Strat(\OO)$ be a minimal stratum. Then $S$ is a closed set, and hence compact.
     \end{coro}
     \begin{proof}
         Let $S$ be a minimal stratum. Assume that $\overline{S} \backslash S \neq \emptyset$. Let $x \in \overline{S} \backslash S \neq \emptyset$, then by the proof of Lemma \ref{G(S) hookrightarrow Gamma_x for x in closure(S)-S} we have that $x$ belongs to a stratum, $S^{'}$, such that $G(S) < G(S^{'})$, in contradiction with $S$ being a minimal stratum. Hence we conclude that $\overline{S} \backslash S = \emptyset$, this is, $\overline{S} = S$, or in other words, $S$ is a closed set.
     \end{proof}

     \begin{prop}\label{if x in closure(S) then S_x in closure S}
         Let $\OO$ be an orbifold and $S \in Strat(\OO)$. If $x \in \overline{S}$ then $S_{x} \subseteq \overline{S}$ where $S_{x}$ is a stratum such that $x \in S_{x}$.
     \end{prop}
     \begin{proof}
         Note that $x \in S_{x} \cap \overline{S}$ so $S_{x} \cap \overline{S} \neq \emptyset$. We will prove that $S_{x} \cap \overline{S}$ is both open and closed in $S_{x}$, then by connectedness of $S_{x}$ we will have that $S_{x} \cap \overline{S} = S_{x}$, or in others words $S_{x} \subseteq \overline{S}$.
         
         Since $\overline{S}$ is closed in $\OO$ we have that $S_{x} \cap \overline{S}$ is closed in $S_{x}$. Now let $U_{y} \subseteq \OO$ be a fundamental chart of $y \in S_{x} \cap \overline{S}$. Then $U_{y} \cap S_{x}$ is an neighborhood of $y$ which is open in $S_{x}$. Additionally, for all $z \in U_{y}$ we have that $\Gamma_{z} \hookrightarrow \Gamma_{y} \cong \Gamma_{x}$ so $U_{y} \cap S_{x} = U_{y} \cap S_{y} \subseteq U_{y} \cap \overline{S_{z}} $. Hence $U_{y} \cap S_{x} \subseteq S_{x} \cap \overline{S}$ and we conclude that 
         $$S_{x} \subseteq \overline{S}$$
     \end{proof}

    \begin{prop}\label{closure(strat) - strat  = union of finite strats}
        Let $S \in Strat(\OO)$, if $\overline{S} \backslash S \neq \emptyset$ then we have 
        $$\overline{S} \backslash S = \bigcup_{i=1}^{k}S_{i}$$
        for $k \in \NN$ and for some $S_{i} \in Strat(\OO)$ such that $G(S) < G(S_{i})$ for $i=1, \ldots , k$.
    \end{prop}
    \begin{proof}
         Take $x \in \overline{S} \backslash S$, by Lemma \ref{G(S) hookrightarrow Gamma_x for x in closure(S)-S} $x$ is contained in a stratum, namely $S^{'}_{x}$, such that $G(S) < G(S^{'}_{x}) \cong \Gamma_{x}$. Hence we have the following:

         $$\overline{S} = S \cup \overline{S} \backslash S \subseteq S \cup \left( \bigcup_{x \in \overline{S} \backslash S} S^{'}_{x} \right) \subseteq S \cup \left( \bigcup_{x \in \overline{S} \backslash S} \overline{S^{'}_{x}} \right)$$

         By Proposition \ref{closure(S) - S is compact} $\overline{S} \backslash S$ is compact, therefore there exists $k \in \NN$ such that

         $$\overline{S} \subseteq S \cup \left( \bigcup_{i = 1}^{k} \overline{S^{'}_{x_{i}}} \right)$$

         Notice that we can repeat the same process of decomposition done with $S$ but now with each $S_{i}$. Since every smooth orbifold is locally homeomorphic to a quotient of an euclidian space by a finite orthogonal subgroup (Proposition \ref{every smooth orbifold is locally homeomorphic to R^n/Gamma}) this process of iterative decomposition is finite. Therefore, we have that $\bigcup_{i=1}^{n}S^{'}_{x_{i}}$ is a union of strata and a union of closures of minimal strata and hence, by Corollary \ref{minimal strata are closed sets}, we conclude that is a union of strata, that is

         $$\overline{S} \subseteq S \cup \left(\bigcup_{i=1}^{n} S_{i} \right)$$

         and by Proposition \ref{if x in closure(S) then S_x in closure S} we have the equality, that is

         $$\overline{S} = S \cup \left(\bigcup_{i=1}^{n} S_{i} \right)$$

         where $S_{i} \in Strat(\OO)$ and $G(S) < G(S_{i})$ for $i=1, \ldots , n$.

         Moreover, we have that $S \cap \left( \bigcup_{i=1}^{k} \overline{S_{i}}\right) = \emptyset$, so we conclude that:
         $$\overline{S} \backslash S = \bigcup_{i=1}^{n} S_{i}$$
    \end{proof}
    
    By Proposition \ref{closure(strat) - strat  = union of finite strats} we have that the set $cch(\OO)$ is finite (see Definition \ref{defi chain and minimal stratum}). As a corollary of the same proposition we obtain a disjoint decomposition of the closure of a stratum.

    \begin{coro}\label{closure strat}
        Let $S \in Strat(\OO)$. Then exists  $k \in \NN$ and $S_{i} \in Strat(\OO)$ satisfying $G(S) < G(S_{i})$ for $i=1, \ldots , k$ such that 
        $$\left\{S, \bigcup_{i=1}^{k}S_{i}\right\}$$
        
        is a disjoint decomposition of $\overline{S}$.
    \end{coro}
    \begin{proof}
        From $\overline{S} = S \cup \overline{S} \backslash S$ and proposition \ref{closure(strat) - strat  = union of finite strats} we have that exists  $k \in \NN$ and $S_{i} \in Strat(\OO)$ satisfying $G(S) < G(S_{i})$ for $i=1, \ldots , k$ such that 
        $$\left\{S, \bigcup_{i=1}^{k}S_{i}\right\}$$
        is a decomposition of $\overline{S}$. Moreover $S \cap \left(\bigcup_{i=1}^{k}S_{i} \right) = \emptyset$ by Remark \ref{strata is disjoint} so the decomposition is disjoint.
    \end{proof}

    As a corollary of Proposition \ref{closure(strat) - strat  = union of finite strats} we get that minimal strata are closed sets, and hence compact. This fact motivates the following lemma.

    \begin{lema}\label{Minimal strat is manifold}
        Let $\OO$ be an orbifold, $S \in Strat(\OO)$ a minimal stratum and $\widetilde{S}$ a lift of $S$. Then $S$ and $\widetilde{S}$ are manifolds and $S \cong \widetilde{S}$.
    \end{lema}
    \begin{proof}
        Take a local chart $(U_{x}, \widetilde{U_{x}}, \Gamma_{x}, \varphi_{x})$ for each $x \in S$. Because $S$ is compact there exists a finite number of tubular neighborhoods that covers $S$. Let $\{U_{x_i}\}$ for $i=1, \ldots , n$ be a finite cover of $S$. We have that:

        $$S \cap U_{x_i} = \quot{\Fix_{\Gamma_{x_i}}(U_{x_i})}{\Gamma_{x_i}}$$

        Hence we conclude that:
        $$S = \bigcup_{i=1}^{n} S \cap U_{x_i} = \bigcup_{i=1}^{n} \quot{\Fix_{\Gamma_{x_i}}(U_{x_i})}{\Gamma_{x_i}} \hspace{20 pt} \widetilde{S} = \bigcup_{i=1}^{n} \widetilde{S \cap U_{x_i}} = \bigcup_{i = 1}^{n} \Fix_{\Gamma_{x_{i}}}(U_{x_{i}})$$

        We conclude then that $S$ and $\widetilde{S}$ are manifolds since the intersection of charts is a chart with local group isomorphic to the local groups of the charts that are intersecting. Moreover, since
        $$ \quot{\Fix_{\Gamma_{x_i}}(U_{x_i})}{\Gamma_{x_i}} \cong \Fix_{\Gamma_{x_{i}}}(U_{x_{i}})$$
        we conclude that $S \cong \widetilde{S}$.
    \end{proof}

    Now we will focus on the problem of extracting a neighbourhood of a minimal strata.

\begin{teor}\label{tubular neighborhood theorem for minimal strats}
Let $\OO$ be an orbifold and $S$ a minimal stratum. There exists a neighborhood of $S$ inside $\OO$ which is a smooth very good orbifold.
\end{teor}

\begin{proof}

Let $S$ be a minimal stratum. Assume first that $S \cap \partial \OO \neq \emptyset$. We claim that $\widetilde{S}$ has a finite open cover of tubular neighborhoods.

For each $x \in S$ take a fundamental chart $(U_{x}, \widetilde{U_{x}}, \Gamma_{x}, \varphi_{x})$ around $x$. Therefore $\bigcup_{x \in S} U_{x}$ is an open cover of $S$. Because $S$ is compact there exists $\{x_{i}\} \subseteq S$ such that 

$$S \subseteq \bigcup_{i=1}^{n}U_{x_{i}}$$

where $x_{i} \in S \cap U_{x_{i}}$ for $i=1, \ldots , n$. 

Let $\widetilde{S \cap U_{x_{i}}}$ be a lift of $S \cap U_{x_{i}}$. By construction we have that:

$$\widetilde{S \cap U_{x_{i}}} = \Fix_{\Gamma_{x_{i}}}(U_{i})$$

Hence $\widetilde{S \cap U_{x_{i}}}$ is a linear variety, and consequently a submanifold of $\widetilde{U_{x_{i}}}$, as in the proof of Lemma \ref{Minimal strat is manifold}.

For each $\widetilde{S  \cap U_{x_{i}}}$ there exists a tubular neighborhood inside $\widetilde{U_{x_{i}}}$. Therefore, we can assume without loss of generality that the open sets $\widetilde{U_{x_{i}}}$ are tubular neighborhoods.

We have the following diagram:

\begin{center}
\begin{tikzcd}
\widetilde{U_{x_{i}}} \arrow[dd, "\pi_{\Gamma_{i}}", two heads]                       &  & \widetilde{U_{x_{i}} \cap U_{x_{j}}} \arrow[dd, "\pi_{\Gamma_i}", two heads] \arrow[rr, "\psi_{j}", hook] \arrow[ll, "\psi_i"', hook']                &  & \widetilde{U_{x_{j}}} \arrow[dd, "\pi_{\Gamma_{j}}", two heads, hook]                 \\
                                                                                      &  &                                                                                                                                                       &  &                                                                                       \\
\quot{\widetilde{U_{x_{i}}}}{\Gamma_{x_i}} \arrow[dd, "\varphi_{i}", two heads, hook] &  & {\quot{\widetilde{U_{x_{i}} \cap U_{x_{j}}}}{\Gamma_{x_{i},x_{j}}}} \arrow[dd, "{\varphi_{i,j}}", two heads, hook] \arrow[ll, hook'] \arrow[rr, hook] &  & \quot{\widetilde{U_{x_{j}}}}{\Gamma_{x_j}} \arrow[dd, "\varphi_{j}", two heads, hook] \\
                                                                                      &  &                                                                                                                                                       &  &                                                                                       \\
U_{x_{i}}                                                                             &  & U_{x_{i}} \cap U_{x_{j}} \arrow[rr, hook] \arrow[ll, hook']                                                                                           &  & U_{x_{j}}                                                                            
\end{tikzcd}
\end{center}

Note that the sets $\widetilde{U_{x_{i}}}$ are manifolds. We can join them to form a manifold that covers $\widetilde{S}$ in a way that will descend to a smooth very good orbifold structure that covers $S$. To this extent we define

$$U:= \bigsqcup_{i=1}^{n} \quot{\widetilde{U_{x_{i}}}}{\sim}$$

where the relation $\sim$ is defined as follows: if $U_{x_{i}} \cap U_{x_{j}} \neq \emptyset$ then $\forall x \in \widetilde{U_{x_{i}} \cap U_{x_{j}}}$ $\psi_{i}(x) \sim \psi_{j}(x)$. By construction $U$ is a manifold without boundary. Moreover, choosing $i \in \{1, \ldots , n\}$ it can be defined a $\Gamma_{x_{i}}$-action on $U$ as follows. First note that every local chart $U_{j}$ we have isomorphisms

\begin{center}
            \begin{tikzcd}
\Gamma_{x_{i}} &  & {\Gamma_{x_{i},x_{j}}} \arrow[rr, "f_{x_{j}}", hook] \arrow[ll, "f_{x_{i}}"', hook'] &  & \Gamma_{x_{j}}
\end{tikzcd}
\end{center}

Now, for every $x \in U$ there exists a local chart $U_{x_{j}}$ such that $x \in U_{x_{j}}$, in this situation we define a $\Gamma_{x_{i}}$ action on $U$ by
$$\gamma \cdot x := (f_{x_{j}} \circ f^{-1}_{x_{i}})(\gamma) \cdot x$$

Because $f_{x_{i}}$ and $f_{x_{j}}$ are isomorphisms the action is well defined. Hence by proposition \ref{M/Gamma is orbifold} we conclude that $\quot{U}{\Gamma_{x_{i}}}$ is a smooth very good orbifold that covers $S$ inside $\OO$ and because the map $\pi_{\Gamma_{x_{1}}} : U \longrightarrow \quot{U}{\Gamma_{x_{1}}}$ is open we have that $\quot{U}{\Gamma_{x_{1}}}$ is topologically open.

If $S \cap \partial \OO \neq \emptyset$ the same proof applies for local charts in $\RR_{+}^{n}$ instead of in $\RR^{n}$.

\end{proof}

\begin{lema}\label{Fix is neat submanifold}
    Let $M$ be a $n$-manifold with boundary and $\Gamma$ a finite group acting locally orthogonal on $M$. Then $\Fix_{\Gamma}(M)$ is a neat submanifold of $M$.
\end{lema}

\begin{proof}
    Take a local chart $W \subseteq M$ around a point in $\partial M$. We can assume that $W$ is of the form

    $$W = W_{0} \times [0,l)$$

    for $l \in \RR_{+}$ and $W_{0} \subseteq \RR^{n-1}$ an open subset. Then we have that
    
    $$\Fix_{\Gamma}(W) = \Fix_{\Gamma_{0}}(W_{0}) \times [0,l)$$
    
    where $\Gamma_{0}$ is a finite group such that $\Gamma_{0} \hookrightarrow O(n-1)$ and whose action on $W_{0}$ is compatible with the action of $\Gamma$ on $M$. Notice that since $W_{0}$ is an open subset of $\RR^{n-1}$, and hence boundariless, we have that
    $$\partial \Bigl(\Fix_{\Gamma_{0}}(W_{0})\Bigr) = \emptyset$$
    
    Furthermore, we have that

    $$\partial \Bigl(\Fix_{\Gamma}(W)\Bigr) = \partial \Bigl(\Fix_{\Gamma_{0}}(W_{0})\Bigr) \times [0,l) \cup \Fix_{\Gamma_{0}}(W_{0}) \times \partial [0,l) = \Fix_{\Gamma_{0}}(W_{0}) \times \{0\}$$

    hence 

    $$\partial \Bigl(\Fix_{\Gamma}(W)\Bigr) = \Fix_{\Gamma}(W) \cap \partial W$$

    We conclude that $\Fix_{\Gamma}(W)$ is neat and therefore $\Fix_{\Gamma}(M)$ is neat.
\end{proof}

\begin{lema}\label{S-epsilon is neat submanifold}
    Let $M$ be a $n$-manifold with boundary and $N$ a neat $m$-submanifold of $M$. Then the set defined as
    $$S_{\varepsilon}(N) : = \{x \in M \mid d(x,N) = \varepsilon\}$$
    is a neat submanifold of $M$ for a sufficiently small $\varepsilon \geq 0$.
\end{lema}
\begin{proof}
    We proceed as in Lemma \ref{Fix is neat submanifold}. Take a local chart $W \subseteq M$ such that $W \cap N \neq \emptyset$. Without loss of generality we can assume that

    $$W = W_{0} \times [0,l) \hspace{20pt}  W \cap N = (W_{0} \cap N_{0} ) \times [0,r)$$

    for $l,r \in \RR_{+}$, $W_{0} \subseteq \RR^{n-1}$ an open subset and $N_{0} \subseteq \RR^{m-1}$ an open subset. We have that
    
    $$W \cap S_{\varepsilon}(N) = \Bigl(W_{0} \cap S_{\varepsilon}(N_{0}) \Bigr)\times [0,r)$$
    
     Notice that since $W_{0}$ is an open set of $\RR^{n-1}$, and hence boundariless, we have that
    $$\partial \Bigl(\Fix_{\Gamma_{0}}(W_{0}) \Bigr)= \emptyset$$
    
    Furthermore, we have that

    $$\Bigl(W \cap S_{\varepsilon}(N) \Bigr) \cap \partial W = \Bigl(\bigl(W_{0} \cap S_{\varepsilon}(N_{0}) \bigr) \times [0,r) \Bigr) \cap \partial W = \Bigl(W_{0} \cap S_{\varepsilon}(N_{0}) \Bigr) \times \{0\}$$

    hence 

    $$\partial \Bigl(W \cap S_{\varepsilon}(N) \Bigr) = \Bigl(W \cap S_{\varepsilon}(N) \Bigr)\cap \partial W$$

    We conclude that $W \cap S_{\varepsilon}(N)$ is neat in $W$ and therefore $S_{\varepsilon}(N)$ is neat in $M$.
\end{proof}

\begin{lema}\label{neat decomposition by minimal strat}
    Let $\OO$ be an orbifold and $S \in Strat(\OO)$ a minimal stratum. There exists a neat decomposition of $\OO$, namely $\{\OO_{1}, \OO_{2}\}$, such that $S \subseteq \OO_{1}$, with $\OO_{1}$ being a very good orbifold, and $S \cap \OO_{2} = \emptyset$. 
\end{lema}

\begin{proof}
    
    By Theorem \ref{tubular neighborhood theorem for minimal strats} we have that $S$ has a neighborhood of the form $\quot{M}{\Gamma}$ for $M$ a manifold and $\Gamma$ a finite group. By Lemma \ref{Fix is neat submanifold} we have that $\Fix_{\Gamma}(M)$ is a neat submanifold of $M$ and by Theorem \ref{neat submanifold has tubular neighborhood} $\Fix_{\Gamma}(M)$ has a tubular neighborhood in $M$. Without loss of generality we can assume that this tubular neighborhood is of the form

    $$N_{\varepsilon}(\Fix_{\Gamma}(M)) := \{x \in M \mid d(x,\Fix_{\Gamma}(M)) < \varepsilon\}$$

    and so 

    $$\overline{N_{\varepsilon}(\Fix_{\Gamma}(M))} = \{x \in M \mid d(x,\Fix_{\Gamma}(M)) \leq \varepsilon\}$$

    where $d$ is the distance function in $M$. By Lemma \ref{S-epsilon is neat submanifold} there exists $\varepsilon^{\prime} \geq 0$ such that $S_{\varepsilon}(\Fix_{\Gamma}(M))$ is neat submanifold of $M$, so we fix $\varepsilon = \varepsilon^{\prime}$.

    Now take
    $$\OO_{1} := \quot{\overline{N_{\varepsilon^{\prime}}(\Fix_{\Gamma}(M))}}{\Gamma} \hspace{20 pt} \OO_{2} := \OO - \quot{N_{\varepsilon^{\prime}}(\Fix_{\Gamma}(M))}{\Gamma}$$

    We claim that $\{\OO_{1}, \OO_{2}\}$ is a neat decomposition of $\OO$. It is clear that $\OO_{1} \cup \OO_{2} = \OO$. Moreover, we have that 
    $$\OO_{1} \cap \OO_{2} = \quot{S_{\varepsilon^{\prime}}(\Fix_{\Gamma}(M))}{\Gamma}$$

    By Lemma \ref{S-epsilon is neat submanifold} $S_{\varepsilon^{\prime}}(\Fix_{\Gamma}(M))$ is neat submanifold of $M$, hence $\quot{S_{\varepsilon^{\prime}}(\Fix_{\Gamma}(M))}{\Gamma}$ is neat suborbifold of $\quot{M}{\Gamma}$. Also it is clear that 

    $$\OO_{1} \cap \OO_{2} \subseteq \partial \OO_{1}, \partial \OO_{2}$$
    so we conclude that $\{\OO_{1}, \OO_{2}\}$ is a neat decompositon of $\OO$ and by construction $\OO_{1}$ is a very good orbifold.
\end{proof}

Now we prove Theorem \ref{big theorem}.

\begin{proof}
    Let $\OO$ be a compact $n$-orbifold. We will prove Theorem \ref{big theorem} by induction on $n$.
    
    Note that, by Remark \ref{big theorem holds for dimension 1}, the result is true for $n=1$. Assume that the result holds for $k$ for some odd $k \in \NN$ and take $n=k+2$. Let $S$ be a minimal stratum of $\OO$, by Lemma \ref{neat decomposition by minimal strat} there exists a decomposition of $\OO$, namely $\{\OO_{1}, \OO_{2}\}$, such that $S \subseteq \OO_{1}$, with $\OO_{1}$ being a very good orbifold, and $S \cap \OO_{2} = \emptyset$. In the same way, we can now take a minimal stratum of $\OO_{2}$ and give a neat decomposition of $\OO_{2}$. Indeed, we can repeat this process until all the strata is subtracted in finitely many steps, obtaining a manifold. Therefore by applying Proposition \ref{lema importante} iteratively case $n=k+2$ is proven and hence by induction Theorem \ref{big theorem} follows.
\end{proof}

\begin{exam}[Euler characteristic of $\quot{D^{3}}{T_{24}}$]
    Let $T_{24}$ be the group of symmetries of the tetrahedron. By Theorem \ref{big theorem} we have that 

    $$\chi \left(\quot{D^{3}}{T_{24}} \right) = \frac{1}{2}\chi \left(\partial \left(\quot{D^{3}}{T_{24}} \right)\right )$$

    Note that $\partial \left(\quot{D^{3}}{T_{24}} \right)$ is homeomorphic to an orbifold with underliying space being $D^{2}$, one corner reflector point of order $2$ and $2$ coorner reflector points of order $3$. Hence by Proposition \ref{Euler characteristic for two-orbifolds} we have that 

    $$\chi \left(\quot{D^{3}}{T_{24}} \right) = \frac{1}{2}\left( \chi(D^{2}) - \frac{1}{2} \left(1 - \frac{1}{2} + 1 - \frac{1}{3} + 1 - \frac{1}{3}\right)\right) = \frac{1}{24}$$

\end{exam}


\thispagestyle{empty}

\textbf{Acknowledgments.} I would like to thank Professor Joan Porti (Universitat Autònoma de Barcelona) for his extensive knowledge and unwavering support during the realization of this work.

\vspace{30pt}

\textsc{Ramon Gallardo Campos} \\
\textit{Email:} \textmd{ramongallardocampos@gmail.com} or \textmd{rgallaca92@alumnes.ub.edu}


\begin{thebibliography}{PBM06}

\bibitem[au222]{caramello}
Francisco C. Caramello~Jr au2.
\newblock Introduction to orbifolds, 2022.

\bibitem[Cho12]{choi}
S.~Choi.
\newblock {\em {Geometric Structures on 2-orbifolds: Exploration of Discrete Symmetry. MSJ Memoirs 27}}.
\newblock Mathematical Society of Japan, 2012.

\bibitem[Hir76]{hirsch-1976}
Morris~W. Hirsch.
\newblock {\em {Differential topology}}.
\newblock 1 1976.

\bibitem[KL14]{AST}
Bruce Kleiner and John Lott.
\newblock Geometrization of three-dimensional orbifolds via {Ricci} flow.
\newblock In {\em Local collapsing, orbifolds, and geometrization}, number 365 in Ast\'erisque. Soci\'et\'e math\'ematique de France, 2014.

\bibitem[Lee12]{lee-2012}
John~M. Lee.
\newblock {\em {Introduction to smooth manifolds}}.
\newblock 1 2012.

\bibitem[Muk15]{mukherjee-2015}
Amiya Mukherjee.
\newblock {\em {Differential topology}}.
\newblock Birkh\"auser, 6 2015.

\bibitem[PBM06]{porti-2006}
Joan Porti, Michel Boileau, and Sylvain Maillot.
\newblock {\em {Three-Dimensional orbifolds and their geometric structures}}.
\newblock 1 2006.

\bibitem[Sat57]{satake-1957}
Ichir\^o Satake.
\newblock {The Gauss-Bonnet Theorem for V-manifolds.}
\newblock {\em Journal of the Mathematical Society of Japan}, 9(4), 10 1957.

\bibitem[Sco83]{Scott}
Peter Scott.
\newblock The geometries of 3-manifolds.
\newblock {\em Bulletin of The London Mathematical Society}, 15:401--487, 1983.

\bibitem[Thu23]{thurston-2023}
William~P. Thurston.
\newblock {\em {The Geometry and Topology of Three-Manifolds}}.
\newblock American Mathematical Society, 6 2023.

\end{thebibliography}
\end{document}